\documentclass[12pt,a4paper]{article}

\usepackage[british]{babel}
\usepackage{amssymb,amsmath,amsthm,amsfonts,mathtools}
\usepackage{hyperref}
\usepackage{cleveref}
\usepackage[margin=2.5cm]{geometry}

\hypersetup{
    colorlinks=true,
    linkcolor=black,
    citecolor=black,
    urlcolor=blue
}

\theoremstyle{plain}
\newtheorem{theorem}{Theorem}[section]
\newtheorem{prop}[theorem]{Proposition}
\newtheorem{corollary}[theorem]{Corollary}
\newtheorem{fact}[theorem]{Fact}

\theoremstyle{definition}

\newtheorem{remark}[theorem]{Remark}

\newcommand{\E}{\mathbb{E}}
\newcommand{\C}{\mathbb{C}}
\newcommand{\R}{\mathbb{R}}
\newcommand{\Tr}{\operatorname{Tr}}
\newcommand{\GL}{\operatorname{GL}}
\newcommand{\one}{\mathbf{1}}
\newcommand{\Sub}{\operatorname{Sub}}

\title{Sidorenko Inequalities for Two-Sided Group Correlation Kernels}
\author{Yuqi Zhao\thanks{Email: \texttt{yuqi.zhao012@gmail.com}}}
\date{}

\begin{document}

\maketitle

\begin{abstract}
Sidorenko's conjecture asserts that every bipartite graph has at least the
expected homomorphism density in every graph of a given edge density. Motivated
by Cayley-type formulations of Sidorenko-type inequalities, we study a
two-sided correlation construction on finite groups.

Let \(\Gamma\) be a finite group and let \(f:\Gamma\to\R\) be a real-valued
function. We define a directed kernel on \(\Gamma\) by
\[
    \mathcal C_f(x,y)
    :=
    \frac{1}{|\Gamma|}
    \sum_{\substack{a_1,a_2\in\Gamma\\ xa_1=a_2y}}
    f(a_1)f(a_2)
    =
    \E_{z\in\Gamma} f(x^{-1}z)f(zy^{-1}).
\]
When \(f=\one_A\), this is the normalized size of the intersection
\(xA\cap Ay\).

We prove that, for every finite directed graph \(F\),
\[
    t(F,\mathcal C_f)
    \geq
    t(\overrightarrow{K_2},\mathcal C_f)^{e(F)}
    =
    \left(\E_{g\in\Gamma}f(g)\right)^{2e(F)}.
\]
Equivalently, if \(W_f^\times(x,y)=f(xy)\) is the directed product Cayley
kernel on \(\Gamma\), then the directed \(1\)-subdivision of every finite
directed graph satisfies the same homomorphism-density lower bound in
\(W_f^\times\).
\end{abstract}

\section{Introduction}

Sidorenko's conjecture is one of the central problems in extremal graph theory. It was proposed independently by
Sidorenko~\cite{sidorenko1993correlation,sidorenko1991inequalities} and
Erd\H{o}s--Simonovits~\cite{erdos1984cube}. In one of its standard forms, it
asserts that every bipartite graph \(H\) and every graph \(G\) satisfy
\[
    t(H,G)\geq t(K_2,G)^{e(H)}.
\]
Equivalently, among all host graphs with a fixed edge density, the random graph
has the smallest possible \(H\)-density.

The conjecture is known for several important families of bipartite graphs.
Sidorenko proved it for trees, even cycles, and complete bipartite graphs
~\cite{sidorenko1993correlation}. Later work verified the conjecture for
hypercubes~\cite{hatami2010graph}, bipartite graphs with one vertex complete to
the other part~\cite{conlon2010approximate}, and strongly tree-decomposable
graphs~\cite{conlon2018some}, among others. Subdivision graphs form another
important class in this context; see, for example,
\cite{conlon2018some,chen2024kohayakawa,im2024sidorenko}.

We shall use the following directed kernel notation. Let \((\Omega,\mu)\) be a
probability space, and let
\[
    W:\Omega\times\Omega\to\R
\]
be a bounded measurable kernel. For a finite directed graph \(F\), define
\begin{equation}\label{eq:directed-density}
    t(F,W)
    :=
    \int_{\Omega^{V(F)}}
    \prod_{uv\in E(F)} W(x_u,x_v)
    \prod_{v\in V(F)} d\mu(x_v).
\end{equation}
Here \(uv\) denotes a directed edge from \(u\) to \(v\). We write
\(\overrightarrow{K_2}\) for the directed graph consisting of one directed edge,
so that
\begin{equation}\label{eq:directed-edge-density}
    t(\overrightarrow{K_2},W)
    =
    \int_{\Omega^2} W(x,y)\,d\mu(x)d\mu(y).
\end{equation}
In the finite setting, which is the setting of this paper, \(\Omega\) is
equipped with the uniform probability measure. This notation also covers the
real-valued kernels that arise below from functions on finite groups.

Group kernels enter Sidorenko-type questions through highly symmetric host
graphs. A reduction of Szegedy~\cite{szegedy2015sparse} shows that Sidorenko's
conjecture can be tested on special bipartite Cayley-type hosts arising from
symmetric groups. In a Cayley form, one is led to kernels of the form
\[
    (x,y)\mapsto a(x^{-1}y),
\]
where \(a\) is a function on a finite group. This motivates the study of
homomorphism densities in algebraically structured kernels on finite groups.

The present paper studies a two-sided analogue of such Cayley-type kernels.
Instead of assigning weight according to a single group element \(x^{-1}y\), we
assign weight by counting weighted solutions of
\[
    xa_1=a_2y.
\]
In the finite setting considered below, every function on \(\Gamma\) is
automatically bounded. Let \(\Gamma\) be a finite group and let
\[
    f:\Gamma\to\R
\]
be a real-valued function. We define the two-sided group correlation kernel
\(\mathcal C_f:\Gamma\times\Gamma\to\R\) by
\begin{equation}\label{eq:two-sided-correlation-kernel}
    \mathcal C_f(x,y)
    :=
    \frac{1}{|\Gamma|}
    \sum_{\substack{a_1,a_2\in\Gamma\\ xa_1=a_2y}}
    f(a_1)f(a_2).
\end{equation}
Writing
\[
    z=xa_1=a_2y,
\]
we get the equivalent form
\begin{equation}\label{eq:common-point-form}
    \mathcal C_f(x,y)
    =
    \E_{z\in\Gamma} f(x^{-1}z)f(zy^{-1}).
\end{equation}
Equivalently, after the change of variables \(a=x^{-1}z\),
\begin{equation}\label{eq:single-variable-form}
    \mathcal C_f(x,y)
    =
    \E_{a\in\Gamma} f(a)f(xay^{-1}).
\end{equation}

The normalization in \eqref{eq:two-sided-correlation-kernel} gives the simple
edge-density identity
\begin{equation}\label{eq:edge-density}
\begin{aligned}
    t(\overrightarrow{K_2},\mathcal C_f)
    &=
    \E_{x,y\in\Gamma}\mathcal C_f(x,y)        \\
    &=
    \E_{x,y,z\in\Gamma} f(x^{-1}z)f(zy^{-1}) \\
    &=
    \left(\E_{g\in\Gamma}f(g)\right)^2.
\end{aligned}
\end{equation}
Indeed, if \(x,y,z\) are independent and uniformly distributed on \(\Gamma\),
then \(x^{-1}z\) and \(zy^{-1}\) are independent and uniformly distributed on
\(\Gamma\).

When \(f=\one_A\) for a subset \(A\subseteq\Gamma\),
\[
    \mathcal C_{\one_A}(x,y)
    =
    \frac{|xA\cap Ay|}{|\Gamma|}.
\]
Thus \(\mathcal C_{\one_A}(x,y)\) is the normalized number of ways to solve
\[
    xa_1=a_2y,
    \qquad a_1,a_2\in A.
\]
For general \(f\), the same expression gives the corresponding weighted count.

Our main theorem is the following.

\begin{theorem}[Main theorem]\label{thm:main}
Let \(\Gamma\) be a finite group, and let
\[
    f:\Gamma\to\R
\]
be a real-valued function. Let \(\mathcal C_f\) be defined by
\eqref{eq:common-point-form}. Then for every finite directed graph \(F\),
\[
    t(F,\mathcal C_f)
    \geq
    t(\overrightarrow{K_2},\mathcal C_f)^{e(F)}.
\]
Equivalently, by \eqref{eq:edge-density},
\[
    t(F,\mathcal C_f)
    \geq
    \left(\E_{g\in\Gamma}f(g)\right)^{2e(F)}.
\]
\end{theorem}

The theorem has an equivalent formulation in terms of directed subdivisions and
the directed product Cayley kernel
\[
    W_f^\times(x,y)=f(xy).
\]
In this form, every directed edge of \(F\) is replaced by a directed path of
length two with the same overall direction. This is often the more natural
formulation from the subdivision point of view, and it is stated precisely in
Section~\ref{sec:subdivision-formulation}.

The indicator-function case gives the following immediate consequence.

\begin{corollary}\label{cor:subset}
Let \(\Gamma\) be a finite group and let \(A\subseteq\Gamma\). Define
\[
    \mathcal C_A(x,y)
    :=
    \frac{|xA\cap Ay|}{|\Gamma|}.
\]
Then for every finite directed graph \(F\),
\[
    t(F,\mathcal C_A)
    \geq
    \left(\frac{|A|}{|\Gamma|}\right)^{2e(F)}.
\]
Equivalently,
\[
    t(F,\mathcal C_A)
    \geq
    t(\overrightarrow{K_2},\mathcal C_A)^{e(F)}.
\]
\end{corollary}

We prove Theorem~\ref{thm:main} by Fourier expansion on the finite group. After
expanding every occurrence of \(f\), the average over each subdivision variable
forces the two representations on the corresponding edge to be contragredient.
For each such edge, the resulting matrix block is positive semidefinite. The
averages over the original vertices give commuting orthogonal projections; the
commutativity comes from the fact that, in each edge tensor factor, at most one
vertex average acts nontrivially. Consequently every Fourier contribution is
non-negative, and the trivial representation contributes exactly
\[
    \left(\E_{g\in\Gamma}f(g)\right)^{2e(F)}.
\]

The paper is organized as follows. In Section~\ref{sec:preliminaries}, we
collect the linear algebra and finite-group Fourier facts used in the proof. In
Section~\ref{sec:proof-main}, we prove Theorem~\ref{thm:main}. In
Section~\ref{sec:subdivision-formulation}, we state the equivalent
product-Cayley formulation for directed \(1\)-subdivisions and explain its
relation to \(1\)-subdivision results for conjugacy-averaged Cayley kernels.

\section{Preliminaries}\label{sec:preliminaries}

\subsection{Linear-algebraic tools}

All matrices in this paper are over \(\C\). For a matrix \(A\), we write
\(A^*\) for its conjugate transpose. If \(A_1,\ldots,A_k\) are square matrices,
then
\[
    \Tr(A_1\otimes\cdots\otimes A_k)
    =
    \prod_{i=1}^k \Tr(A_i).
\]
We shall also use the following tensor-product identity: if, for each \(i\),
the product \(A_iB_i\) is defined, then
\begin{equation}\label{eq:tensor-product-multiplication}
    \left(\bigotimes_{i=1}^k A_i\right)
    \left(\bigotimes_{i=1}^k B_i\right)
    =
    \bigotimes_{i=1}^k A_iB_i.
\end{equation}

A square matrix \(A\) is called positive semidefinite if \(A=A^*\) and
\[
    \langle Av,v\rangle\geq 0
    \qquad\text{for every vector }v.
\]
A matrix \(P\) is called an orthogonal projection if
\[
    P^2=P
    \qquad\text{and}\qquad
    P=P^*.
\]
Every orthogonal projection is positive semidefinite.

\begin{fact}\label{fact:linear-algebra}
The following facts hold.
\begin{enumerate}
    \item If \(A\) and \(B\) are positive semidefinite matrices of the same size,
    then
    \[
        \Tr(AB)\geq 0.
    \]

    \item If \(P_1,\ldots,P_k\) are pairwise commuting orthogonal projections,
    then
    \[
        P_1P_2\cdots P_k
    \]
    is an orthogonal projection.

    \item If \(A_1,\ldots,A_k\) are positive semidefinite matrices, then
    \[
        A_1\otimes\cdots\otimes A_k
    \]
    is positive semidefinite.
\end{enumerate}
\end{fact}

\begin{proof}
For the first assertion, write
\[
    \Tr(AB)=\Tr(A^{1/2}BA^{1/2}).
\]
The matrix \(A^{1/2}BA^{1/2}\) is positive semidefinite, so its trace is
non-negative.

For the second assertion, the product \(P=P_1\cdots P_k\) satisfies
\[
    P^2=P
    \qquad\text{and}\qquad
    P^*=P,
\]
because the projections commute pairwise.

For the third assertion, write \(A_i=B_iB_i^*\) for each \(i\). Then
\[
    A_1\otimes\cdots\otimes A_k
    =
    (B_1\otimes\cdots\otimes B_k)
    (B_1\otimes\cdots\otimes B_k)^*,
\]
which is positive semidefinite.
\end{proof}

\begin{fact}[Averaging projection]\label{fact:averaging-projection}
Let \(\pi:\Gamma\to\GL(V)\) be a finite-dimensional unitary representation of a
finite group \(\Gamma\). Then
\[
    \E_{g\in\Gamma}\pi(g)
\]
is the orthogonal projection onto the invariant subspace
\[
    V^\Gamma
    =
    \{v\in V:\pi(g)v=v\text{ for every }g\in\Gamma\}.
\]
\end{fact}

\begin{proof}
Let
\[
    P:=\E_{g\in\Gamma}\pi(g).
\]
Since \(\pi\) is unitary and \(g\mapsto g^{-1}\) is a bijection of \(\Gamma\),
we have
\[
    P^*
    =
    \E_{g\in\Gamma}\pi(g)^*
    =
    \E_{g\in\Gamma}\pi(g^{-1})
    =
    P.
\]
Moreover,
\[
    P^2
    =
    \E_{g,h\in\Gamma}\pi(gh)
    =
    \E_{u\in\Gamma}\pi(u)
    =
    P.
\]
Thus \(P\) is an orthogonal projection. Its image is exactly the subspace fixed
by \(\Gamma\).
\end{proof}

\subsection{Fourier analysis on finite groups}

Let \(\Gamma\) be a finite group. All representations are finite-dimensional
complex representations. We choose all irreducible representations to be
unitary, and we fix a set \(\widehat\Gamma\) of representatives of the
irreducible unitary representations of \(\Gamma\). For \(\rho\in\widehat\Gamma\),
write \(V_\rho\) for the representation space and
\[
    d_\rho:=\dim V_\rho.
\]
The contragredient representation of \(\rho\) is denoted by \(\rho^\vee\). We
choose the representatives so that contragredients are included in
\(\widehat\Gamma\). When the chosen representative of the contragredient class is
only unitarily equivalent to the literal dual representation, we fix such an
identification once and for all. Thus notation such as \(\sigma=\rho^\vee\)
below is understood with this convention.

For \(f:\Gamma\to\C\), define its Fourier coefficient at \(\rho\) by
\begin{equation}\label{eq:FT-def}
    \widehat f(\rho)
    :=
    \E_{g\in\Gamma} f(g)\rho(g)^{-1}.
\end{equation}
With this normalization, Fourier inversion gives
\begin{equation}\label{eq:Fourier-inversion}
    f(g)
    =
    \sum_{\rho\in\widehat\Gamma}
    d_\rho\,
    \Tr\!\left(\widehat f(\rho)\rho(g)\right)
    \qquad(g\in\Gamma).
\end{equation}

We shall use the following consequence of Schur's lemma; see, for example,
\cite{serre1977linear}.

\begin{fact}[Tensor averages of irreducible representations]\label{fact:tensor-average-dual}
Let \(\rho,\sigma\in\widehat\Gamma\). Define
\[
    \Pi_{\rho,\sigma}
    :=
    \E_{g\in\Gamma}\rho(g)\otimes\sigma(g).
\]
Then \(\Pi_{\rho,\sigma}\) is an orthogonal projection. Moreover,
\[
    \Pi_{\rho,\sigma}\neq 0
\]
only if
\[
    \sigma\simeq \rho^\vee.
\]
With the above convention for contragredients, the only nonzero case is written
as \(\sigma=\rho^\vee\).
\end{fact}

\begin{proof}
The map
\[
    g\mapsto \rho(g)\otimes\sigma(g)
\]
is a unitary representation of \(\Gamma\), so \(\Pi_{\rho,\sigma}\) is an
orthogonal projection by Fact~\ref{fact:averaging-projection}. Its image is the
invariant subspace of \(V_\rho\otimes V_\sigma\). By Schur's lemma, this
invariant subspace is nonzero if and only if \(V_\sigma\) is isomorphic to
\(V_\rho^\vee\).
\end{proof}

\section{Proof of the main theorem}\label{sec:proof-main}

We first prove a product-Cayley subdivision form of the desired inequality. This
form will also be used in Section~\ref{sec:subdivision-formulation}. For a
directed edge \(e\), write \(e^-\) and \(e^+\) for its tail and head.

\begin{prop}[Subdivision form]\label{prop:subdivision-form}
Let \(F=(V,E)\) be a finite directed graph, and let
\[
    f:\Gamma\to\R
\]
be a real-valued function. Then
\begin{equation}\label{eq:subdivision-form}
    \E_{x:V\to\Gamma,\ z:E\to\Gamma}
    \prod_{e\in E}
    f(x_{e^-}z_e)f(z_e x_{e^+})
    \geq
    \left(\E_{g\in\Gamma}f(g)\right)^{2e(F)}.
\end{equation}
\end{prop}
\begin{proof}
Write
\[
    m:=e(F).
\]
If \(m=0\), then both sides of \eqref{eq:subdivision-form} are equal to \(1\).
Assume henceforth that \(m\geq 1\).

For every edge \(e\in E\), expand the two factors
\[
    f(x_{e^-}z_e),
    \qquad
    f(z_e x_{e^+})
\]
using the Fourier inversion formula \eqref{eq:Fourier-inversion}. This gives a
sum over choices of irreducible representations
\[
    \boldsymbol\rho=(\rho_e)_{e\in E},
    \qquad
    \boldsymbol\sigma=(\sigma_e)_{e\in E}.
\]
For such a pair, set
\[
    \mathcal V_{\boldsymbol\rho,\boldsymbol\sigma}
    :=
    \bigotimes_{e\in E}
    \left(V_{\rho_e}\otimes V_{\sigma_e}\right).
\]

For \(\rho,\sigma\in\widehat\Gamma\), define the edge block
\begin{equation}\label{eq:edge-block-general}
    A_{\rho,\sigma}
    :=
    (I_\rho\otimes \widehat f(\sigma))
    \Pi_{\rho,\sigma}
    (\widehat f(\rho)\otimes I_\sigma),
\end{equation}
where
\[
    \Pi_{\rho,\sigma}
    :=
    \E_{g\in\Gamma}\rho(g)\otimes\sigma(g).
\]
By Fact~\ref{fact:tensor-average-dual},
\[
    A_{\rho,\sigma}=0
\]
unless \(\sigma=\rho^\vee\).

We next define the vertex projections. For \(v\in V\) and \(e\in E\), set
\[
    M^-(v,e):=
    \begin{cases}
    1,&\text{if }e^-=v,\\
    0,&\text{otherwise,}
    \end{cases}
    \qquad
    M^+(v,e):=
    \begin{cases}
    1,&\text{if }e^+=v,\\
    0,&\text{otherwise.}
    \end{cases}
\]
For each fixed edge \(e\), the column \((M^-(v,e))_{v\in V}\) has exactly one
\(1\), and the column \((M^+(v,e))_{v\in V}\) has exactly one \(1\).

For each \(v\in V\), define an operator \(P_v\) on
\(\mathcal V_{\boldsymbol\rho,\boldsymbol\sigma}\) by
\begin{equation}\label{eq:vertex-projection}
    P_v
    :=
    \E_{g\in\Gamma}
    \bigotimes_{e\in E}
    \left(
        \rho_e(g^{M^-(v,e)})
        \otimes
        \sigma_e(g^{M^+(v,e)})
    \right),
\end{equation}
where \(g^0\) denotes the identity element of \(\Gamma\), so that
\(\rho_e(g^0)=I_{\rho_e}\) and \(\sigma_e(g^0)=I_{\sigma_e}\). The integrand in
\eqref{eq:vertex-projection} is a unitary representation of \(\Gamma\). Hence
\(P_v\) is an orthogonal projection by Fact~\ref{fact:averaging-projection}.

Moreover, the projections \(P_v\) commute with one another. Indeed, if
\(v\neq v'\), then for every edge \(e\),
\[
    M^-(v,e)M^-(v',e)=0,
    \qquad
    M^+(v,e)M^+(v',e)=0.
\]
Thus, in each tensor factor \(V_{\rho_e}\) or \(V_{\sigma_e}\), at most one of
the two pointwise operators appearing in the definitions of \(P_v\) and
\(P_{v'}\) is non-identity. Hence the integrands commute pointwise, and after
averaging over the two group variables we obtain
\[
    P_vP_{v'}=P_{v'}P_v.
\]
Therefore
\begin{equation}\label{eq:PF-definition}
    P_F(\boldsymbol\rho,\boldsymbol\sigma)
    :=
    \prod_{v\in V}P_v
\end{equation}
is an orthogonal projection by Fact~\ref{fact:linear-algebra}. Since the factors
\(P_v\) commute, the product is independent of the order of the vertices.

We now compute the Fourier expansion. For a fixed edge \(e\), cyclicity of
trace gives
\begin{align}
&\Tr\!\left(\widehat f(\rho_e)\rho_e(x_{e^-}z_e)\right)
 \Tr\!\left(\widehat f(\sigma_e)\sigma_e(z_e x_{e^+})\right)
 \notag\\
&\qquad =
\Tr\!\left[
    \left(
        \rho_e(z_e)\widehat f(\rho_e)
        \otimes
        \widehat f(\sigma_e)\sigma_e(z_e)
    \right)
    \left(
        \rho_e(x_{e^-})
        \otimes
        \sigma_e(x_{e^+})
    \right)
\right].
\label{eq:edge-trace-rewrite}
\end{align}
Averaging the \(z_e\)-dependent factor in \eqref{eq:edge-trace-rewrite} gives
\[
\begin{aligned}
&\E_{z_e\in\Gamma}
    \rho_e(z_e)\widehat f(\rho_e)
    \otimes
    \widehat f(\sigma_e)\sigma_e(z_e)        \\
&\qquad =
    (I_{\rho_e}\otimes\widehat f(\sigma_e))
    \left(
        \E_{z_e\in\Gamma}
        \rho_e(z_e)\otimes\sigma_e(z_e)
    \right)
    (\widehat f(\rho_e)\otimes I_{\sigma_e}) \\
&\qquad =
    A_{\rho_e,\sigma_e}.
\end{aligned}
\]

Using \eqref{eq:tensor-product-multiplication} and
\[
    \prod_{e\in E}\Tr(T_e)
    =
    \Tr\!\left(\bigotimes_{e\in E}T_e\right),
\]
the \((\boldsymbol\rho,\boldsymbol\sigma)\)-summand, after pulling out the
coefficient
\[
    \prod_{e\in E}d_{\rho_e}d_{\sigma_e},
\]
has the following form before averaging over the variables:
\[
\Tr\!\left[
    \left(
        \bigotimes_{e\in E}
        \left(
            \rho_e(z_e)\widehat f(\rho_e)
            \otimes
            \widehat f(\sigma_e)\sigma_e(z_e)
        \right)
    \right)
    Q(x)
\right],
\]
where
\[
    Q(x)
    :=
    \bigotimes_{e\in E}
    \left(
        \rho_e(x_{e^-})\otimes\sigma_e(x_{e^+})
    \right).
\]
The variables \(z_e\) are independent, and \(Q(x)\) is independent of them.
Therefore, by linearity of trace and the preceding computation,
\[
\begin{aligned}
&\E_{z:E\to\Gamma}
\Tr\!\left[
    \left(
        \bigotimes_{e\in E}
        \left(
            \rho_e(z_e)\widehat f(\rho_e)
            \otimes
            \widehat f(\sigma_e)\sigma_e(z_e)
        \right)
    \right)
    Q(x)
\right]  \\
&\qquad =
\Tr\!\left[
    \left(
        \bigotimes_{e\in E}A_{\rho_e,\sigma_e}
    \right)
    Q(x)
\right].
\end{aligned}
\]

It remains to average \(Q(x)\) over the vertex variables. For \(v\in V\) and
\(g\in\Gamma\), write
\[
    R_v(g)
    :=
    \bigotimes_{e\in E}
    \left(
        \rho_e(g^{M^-(v,e)})
        \otimes
        \sigma_e(g^{M^+(v,e)})
    \right).
\]
Then
\[
    Q(x)
    =
    \prod_{v\in V} R_v(x_v).
\]
Indeed, on the tensor factor \(V_{\rho_e}\), only the tail vertex \(e^-\) acts
nontrivially, and on the tensor factor \(V_{\sigma_e}\), only the head vertex
\(e^+\) acts nontrivially. The factors \(R_v(x_v)\) commute pairwise for the
same reason. Hence, by independence of the variables \(x_v\),
\[
\begin{aligned}
    \E_{x:V\to\Gamma}Q(x)
    &=
    \prod_{v\in V}\E_{g\in\Gamma}R_v(g)  \\
    &=
    \prod_{v\in V}P_v  \\
    &=
    P_F(\boldsymbol\rho,\boldsymbol\sigma).
\end{aligned}
\]
Consequently,
\begin{equation}\label{eq:full-fourier-expansion}
\begin{aligned}
&\E_{x:V\to\Gamma,\ z:E\to\Gamma}
    \prod_{e\in E}
    f(x_{e^-}z_e)f(z_e x_{e^+}) \\
&\qquad =
\sum_{\boldsymbol\rho,\boldsymbol\sigma\in\widehat\Gamma^{E}}
\left(
    \prod_{e\in E}d_{\rho_e}d_{\sigma_e}
\right)
\Tr\!\left[
    \left(
        \bigotimes_{e\in E}A_{\rho_e,\sigma_e}
    \right)
    P_F(\boldsymbol\rho,\boldsymbol\sigma)
\right].
\end{aligned}
\end{equation}

Only dual pairs contribute. For \(\rho\in\widehat\Gamma\), put
\[
    \Pi_\rho:=\Pi_{\rho,\rho^\vee}
    =
    \E_{g\in\Gamma}\rho(g)\otimes\rho^\vee(g),
\]
and
\begin{equation}\label{eq:edge-block-dual}
    A_\rho
    :=
    A_{\rho,\rho^\vee}
    =
    (I_\rho\otimes \widehat f(\rho^\vee))
    \Pi_\rho
    (\widehat f(\rho)\otimes I_{\rho^\vee}).
\end{equation}
We claim that \(A_\rho\) is positive semidefinite. For every \(h\in\Gamma\),
\begin{align*}
    \Pi_\rho(\rho(h^{-1})\otimes I_{\rho^\vee})
    &=
    \E_{x\in\Gamma}
    \rho(xh^{-1})\otimes\rho^\vee(x)\\
    &=
    \E_{y\in\Gamma}
    \rho(y)\otimes\rho^\vee(yh)\\
    &=
    \Pi_\rho(I_\rho\otimes\rho^\vee(h)),
\end{align*}
where we used the change of variables \(y=xh^{-1}\). Averaging this identity
with coefficient \(f(h)\), and using the definition \eqref{eq:FT-def}, gives
\begin{equation}\label{eq:intertwining}
    \Pi_\rho(\widehat f(\rho)\otimes I_{\rho^\vee})
    =
    \Pi_\rho(I_\rho\otimes\widehat f(\rho^\vee)^*).
\end{equation}
Here we used that \(f\) is real-valued. More explicitly, since
\(\rho^\vee\) is unitary,
\[
\begin{aligned}
    \widehat f(\rho^\vee)^*
    &=
    \left(
        \E_{h\in\Gamma}f(h)\rho^\vee(h)^{-1}
    \right)^*  \\
    &=
    \E_{h\in\Gamma}f(h)\rho^\vee(h).
\end{aligned}
\]
This is the only point in the proof where real-valuedness of \(f\) is used.

By \eqref{eq:edge-block-dual} and \eqref{eq:intertwining},
\[
\begin{aligned}
    A_\rho
    &=
    (I_\rho\otimes \widehat f(\rho^\vee))
    \Pi_\rho
    (I_\rho\otimes \widehat f(\rho^\vee)^*)\\
    &=
    \left[(I_\rho\otimes \widehat f(\rho^\vee))\Pi_\rho\right]
    \left[(I_\rho\otimes \widehat f(\rho^\vee))\Pi_\rho\right]^*,
\end{aligned}
\]
where we used that \(\Pi_\rho\) is an orthogonal projection by
Fact~\ref{fact:tensor-average-dual}. Hence \(A_\rho\) is positive semidefinite.

Since \(A_{\rho,\sigma}=0\) unless \(\sigma=\rho^\vee\), the expansion
\eqref{eq:full-fourier-expansion} reduces to
\begin{equation}\label{eq:dual-fourier-expansion}
\begin{aligned}
&\E_{x:V\to\Gamma,\ z:E\to\Gamma}
    \prod_{e\in E}
    f(x_{e^-}z_e)f(z_e x_{e^+}) \\
&\qquad =
\sum_{\boldsymbol\rho\in\widehat\Gamma^{E}}
\left(
    \prod_{e\in E}d_{\rho_e}^{2}
\right)
\Tr\!\left[
    \left(
        \bigotimes_{e\in E}A_{\rho_e}
    \right)
    P_F(\boldsymbol\rho)
\right],
\end{aligned}
\end{equation}
where
\[
    \boldsymbol\rho^\vee:=(\rho_e^\vee)_{e\in E},
    \qquad
    P_F(\boldsymbol\rho)
    :=
    P_F(\boldsymbol\rho,\boldsymbol\rho^\vee).
\]
For each edge \(e\), the matrix \(A_{\rho_e}\) is positive semidefinite. Hence
\[
    \bigotimes_{e\in E}A_{\rho_e}
\]
is positive semidefinite by Fact~\ref{fact:linear-algebra}. The operator
\(P_F(\boldsymbol\rho)\) is an orthogonal projection, hence positive
semidefinite. Therefore, again by Fact~\ref{fact:linear-algebra},
\[
    \Tr\!\left[
    \left(
        \bigotimes_{e\in E}A_{\rho_e}
    \right)
    P_F(\boldsymbol\rho)
    \right]
    \geq 0.
\]
Thus every Fourier contribution in \eqref{eq:dual-fourier-expansion} is
non-negative.

It remains to identify the trivial-representation contribution. Let \(\one\)
denote the trivial representation of \(\Gamma\). Then
\[
    d_{\one}=1,
    \qquad
    \widehat f(\one)=\E_{g\in\Gamma}f(g),
    \qquad
    \Pi_{\one}=1.
\]
Thus
\[
    A_{\one}
    =
    \left(\E_{g\in\Gamma}f(g)\right)^2.
\]
If \(\rho_e=\one\) for every \(e\in E\), then the whole tensor space is
one-dimensional and \(P_F(\boldsymbol\rho)=1\). Hence the corresponding summand
in \eqref{eq:dual-fourier-expansion} is exactly
\[
    \left(\E_{g\in\Gamma}f(g)\right)^{2m}.
\]
Since all other summands are non-negative, we obtain
\[
    \E_{x:V\to\Gamma,\ z:E\to\Gamma}
    \prod_{e\in E}
    f(x_{e^-}z_e)f(z_e x_{e^+})
    \geq
    \left(\E_{g\in\Gamma}f(g)\right)^{2m}.
\]
As \(m=e(F)\), this proves the proposition.
\end{proof}

We now prove Theorem~\ref{thm:main}.

\begin{proof}[Proof of Theorem~\ref{thm:main}]
By definition of \(\mathcal C_f\),
\[
\begin{aligned}
    t(F,\mathcal C_f)
    &=
    \E_{x:V(F)\to\Gamma}
    \prod_{e\in E(F)}
    \mathcal C_f(x_{e^-},x_{e^+})\\
    &=
    \E_{x:V(F)\to\Gamma,\ z:E(F)\to\Gamma}
    \prod_{e\in E(F)}
    f(x_{e^-}^{-1}z_e)f(z_e x_{e^+}^{-1}).
\end{aligned}
\]
Make the change of variables
\[
    u_v=x_v^{-1}
    \qquad(v\in V(F)).
\]
Since inversion preserves the uniform measure on \(\Gamma\), the last
expression becomes
\[
    \E_{u:V(F)\to\Gamma,\ z:E(F)\to\Gamma}
    \prod_{e\in E(F)}
    f(u_{e^-}z_e)f(z_eu_{e^+}).
\]
Applying Proposition~\ref{prop:subdivision-form}, we get
\[
    t(F,\mathcal C_f)
    \geq
    \left(\E_{g\in\Gamma}f(g)\right)^{2e(F)}.
\]
On the other hand, by \eqref{eq:edge-density},
\[
    t(\overrightarrow{K_2},\mathcal C_f)
    =
    \left(\E_{g\in\Gamma}f(g)\right)^2.
\]
Therefore
\[
    t(F,\mathcal C_f)
    \geq
    t(\overrightarrow{K_2},\mathcal C_f)^{e(F)}.
\]
This proves the theorem.
\end{proof}

\section{The product Cayley formulation}\label{sec:subdivision-formulation}

We now record the product Cayley formulation of Theorem~\ref{thm:main}, which
is often the more natural formulation from the subdivision viewpoint.

Let \(F=(V,E)\) be a finite directed graph. Its directed \(1\)-subdivision,
denoted by
\[
    \Sub_1^{\to}(F),
\]
is the directed graph with vertex set
\[
    V\sqcup E,
\]
where each edge \(e\in E\) is regarded as a new subdivision vertex, and with
directed edges
\[
    e^-\longrightarrow e,
    \qquad
    e\longrightarrow e^+
\]
for every \(e\in E\). Thus every directed edge of \(F\) is replaced by a
directed path of length two with the same overall direction.

For a real-valued function \(f:\Gamma\to\R\), define the directed product
Cayley kernel
\[
    W_f^\times:\Gamma\times\Gamma\to\R
\]
by
\begin{equation}\label{eq:product-cayley-kernel}
    W_f^\times(x,y):=f(xy).
\end{equation}
Then
\begin{equation}\label{eq:product-cayley-edge-density}
    t(\overrightarrow{K_2},W_f^\times)
    =
    \E_{x,y\in\Gamma}f(xy)
    =
    \E_{g\in\Gamma}f(g).
\end{equation}

Counting \(\Sub_1^\to(F)\) in \(W_f^\times\), a map from the old vertices
\(V\) to \(\Gamma\) will be denoted by \(x:V\to\Gamma\), and a map from the
subdivision vertices \(E\) to \(\Gamma\) will be denoted by \(z:E\to\Gamma\).
By \eqref{eq:product-cayley-kernel},
\begin{equation}\label{eq:subdivision-product-density}
\begin{aligned}
    t(\Sub_1^{\to}(F),W_f^\times)
    &=
    \E_{x:V\to\Gamma,\ z:E\to\Gamma}
    \prod_{e\in E}
    W_f^\times(x_{e^-},z_e)
    W_f^\times(z_e,x_{e^+})\\
    &=
    \E_{x:V\to\Gamma,\ z:E\to\Gamma}
    \prod_{e\in E}
    f(x_{e^-}z_e)f(z_ex_{e^+}).
\end{aligned}
\end{equation}
Thus Proposition~\ref{prop:subdivision-form} gives
\[
    t(\Sub_1^{\to}(F),W_f^\times)
    \geq
    \left(\E_{g\in\Gamma}f(g)\right)^{2e(F)}.
\]
Since
\[
    e(\Sub_1^{\to}(F))=2e(F),
\]
we obtain
\begin{equation}\label{eq:product-subdivision-sidorenko}
    t(\Sub_1^{\to}(F),W_f^\times)
    \geq
    t(\overrightarrow{K_2},W_f^\times)^{e(\Sub_1^{\to}(F))}.
\end{equation}

\begin{corollary}\label{cor:product-cayley-subdivision}
Let \(F\) be a finite directed graph, let \(\Gamma\) be a finite group, and let
\(f:\Gamma\to\R\). Define \(W_f^\times(x,y)=f(xy)\). Then
\[
    t(\Sub_1^{\to}(F),W_f^\times)
    \geq
    t(\overrightarrow{K_2},W_f^\times)^{e(\Sub_1^{\to}(F))}.
\]
\end{corollary}\begin{remark}\label{rmk:relation-to-conjugacy-averaged-subdivision}
The product Cayley formulation contains the conjugacy-averaged Cayley
formulation for \(1\)-subdivisions as a special case.

Let \(G=(L,R;E(G))\) be a finite bipartite graph. For a function
\(b:\Gamma\to\R\), write
\[
    t_{\mathrm{Cay}}(G;\Gamma,b)
    :=
    \E_{\alpha:L\to\Gamma,\ \beta:R\to\Gamma}
    \prod_{\ell r\in E(G)}
    b(\alpha(\ell)^{-1}\beta(r)).
\]
This is the homomorphism density of \(G\) in the bipartite Cayley kernel
\[
    K_b(x,y):=b(x^{-1}y).
\]
In particular,
\[
    t_{\mathrm{Cay}}(K_2;\Gamma,b)
    =
    \E_{x,y\in\Gamma}b(x^{-1}y)
    =
    \E_{g\in\Gamma}b(g).
\]

Let \(H_0=(V_0,E_0)\) be an undirected graph, and orient the edges of \(H_0\)
arbitrarily to obtain a directed graph \(F\). Let
\[
    H=\Sub(H_0)
\]
be the usual undirected \(1\)-subdivision of \(H_0\), regarded as a bipartite
graph with parts \(V_0\) and \(E_0\). Suppose that \(b:\Gamma\to\R\) is a class
function. Applying Corollary~\ref{cor:product-cayley-subdivision} to \(f=b\)
gives
\[
\begin{aligned}
    t(\Sub_1^{\to}(F),W_b^\times)
    =
    \E_{x:V_0\to\Gamma,\ z:E_0\to\Gamma}
    \prod_{e\in E_0}
    b(x_{e^-}z_e)b(z_e x_{e^+}).
\end{aligned}
\]
After the change of variables
\[
    x_v=\phi(v)^{-1}
    \qquad(v\in V_0),
\]
this becomes
\[
    \E_{\phi:V_0\to\Gamma,\ z:E_0\to\Gamma}
    \prod_{e\in E_0}
    b(\phi(e^-)^{-1}z_e)b(z_e\phi(e^+)^{-1}).
\]
Since \(b\) is a class function,
\[
    b(z_e\phi(e^+)^{-1})
    =
    b(\phi(e^+)^{-1}z_e),
\]
because
\[
    z_e\phi(e^+)^{-1}
    =
    \phi(e^+)\bigl(\phi(e^+)^{-1}z_e\bigr)\phi(e^+)^{-1}.
\]
Therefore
\[
\begin{aligned}
    t(\Sub_1^{\to}(F),W_b^\times)
    &=
    \E_{\phi:V_0\to\Gamma,\ z:E_0\to\Gamma}
    \prod_{e\in E_0}
    b(\phi(e^-)^{-1}z_e)b(\phi(e^+)^{-1}z_e)\\
    &=
    t_{\mathrm{Cay}}(\Sub(H_0);\Gamma,b).
\end{aligned}
\]
Thus Corollary~\ref{cor:product-cayley-subdivision} implies
\[
    t_{\mathrm{Cay}}(\Sub(H_0);\Gamma,b)
    \geq
    \left(\E_{g\in\Gamma}b(g)\right)^{2e(H_0)}.
\]
Since
\[
    e(\Sub(H_0))=2e(H_0)
\]
and
\[
    t_{\mathrm{Cay}}(K_2;\Gamma,b)=\E_{g\in\Gamma}b(g),
\]
this can be rewritten as
\[
    t_{\mathrm{Cay}}(\Sub(H_0);\Gamma,b)
    \geq
    t_{\mathrm{Cay}}(K_2;\Gamma,b)^{e(\Sub(H_0))}.
\]

Finally, for an arbitrary function \(a:\Gamma\to\R\), define its conjugacy-class
average by
\[
    a^{\mathrm{cl}}(g)
    :=
    \E_{h\in\Gamma} a(hgh^{-1}).
\]
Then \(a^{\mathrm{cl}}\) is a class function. Taking \(b=a^{\mathrm{cl}}\) in
the preceding inequality gives
\[
    t_{\mathrm{Cay}}(\Sub(H_0);\Gamma,a^{\mathrm{cl}})
    \geq
    t_{\mathrm{Cay}}(K_2;\Gamma,a^{\mathrm{cl}})^{e(\Sub(H_0))}.
\]
This is precisely the \(1\)-subdivision theorem for conjugacy-averaged Cayley
kernels, namely \cite[Theorem~1.4]{zhao2026conjugacy}.
\end{remark}

\section*{Declaration of Generative AI and AI-assisted technologies in the writing process}

During the preparation of this manuscript, the author used ChatGPT for language
polishing, editing assistance, and improving readability. The author has
reviewed the AI-assisted text and takes full responsibility for the content of
this manuscript. All mathematical ideas, statements, proofs, and final
verification are the author's own.

\bibliographystyle{abbrv}
\bibliography{ref}

\end{document}